\documentclass[10pt,amscd]{amsart}
\usepackage{latexsym}
\usepackage{amsmath}
\usepackage{amssymb}
\usepackage[all]{xy}
\newcommand{\Q}{\mathbb{Q}}
\newcommand{\K}{\mathbb{H}}
\newcommand{\C}{\mathbb{C}}

\newtheorem{thm}{Theorem }[section]
\newtheorem{prop}[thm]{Proposition}
\newtheorem{lem}[thm]{Lemma}
\newtheorem{claim}[thm]{Claim}
\newtheorem{defi}[thm]{Definition}
\newtheorem{cor}[thm]{Corollary}
\newtheorem{exmp}[thm]{Example}
\newtheorem{rem}[thm]{Remark}

\newtheorem{ques}[thm]{Question}

\begin{document}
\title{Sullivan minimal  models  of classifying spaces
 for non-formal spaces of small rank}
\author{Hirokazu Nishinobu  and Toshihiro Yamaguchi}
\footnote[0]{2010MSC: 55P62,55R15}
\date{}
\address{Kochi University, 2-5-1, Kochi,780-8520, JAPAN}
\email{cosmo51mutta@yahoo.co.jp}

\address{Kochi University, 2-5-1, Kochi,780-8520, JAPAN}
\email{tyamag@kochi-u.ac.jp}
\maketitle

\begin{abstract}
We consider certain rational homotopical conditions of simly connected CW complex 
$X$ such that   the  rational cohomology of  the  classifying space $Baut_1X$
 for fibrations with two-stage  fibre $X$ is (not) free.
First, we consider when  is  $Baut_1X$   a rational factor of $Baut_1(X\times S^n)$ for an odd-integer $n$
and observe for a non-formal elliptic space  $X$ of rank 3.
Second,  we compute the  Sullivan minimal models of $Baut_1X$ when $X$ are certain non-formal pure spaces
of  rank 5.
\end{abstract}

\section{Introduction}

Let $X$ be a simply connected finite CW complex 
and 
$aut_1X$ the indentity component of self-homotopy 
equivalences of $X$.
The Dold-Lashof classifying space \cite{DL}, $Baut_1X$, is the 
classifying space for orientable fibrations with
fiber the homotopy type of $X$.
In 1968, J.Milnor \cite{Mi} showed that, when $X=S^n$,
$ H^*(Baut_1X;\Q )\cong \Q [v]$ 
where $|v|=2n$ if $n$ is even and $|v|=n+1$ if $n$ is odd. 
Let $M(X)$ be the Sullivan minimal model \cite{S} of $X$ and $DerM(X)$ the DGL(differential graded Lie algebra) of the  negative derivations on $M(X)$ (see \S 2 below).
In 1977, D.Sullivan \cite{S} indicated that  $DerM(X)$ determines 
the rational homotopy type of $ Baut_1X$, that is, 
$DerM(X)$ is a  Quillen's DGL-model \cite{Q}.
We are interested in the rational homotopical properties of it.
(See \cite{LS} as a recent resarch of this direction.)
It seems to be   very difficult to expect understanding them immediately  from the Sullivan model itself
without the  direct computation (see Theorem \ref{dermodel} below) of derivations of $M(X)$ et al. (cf. \cite{SS}, \cite{G}).
Indeed, it is complicated even when $X$ is a product of spheres \cite{s2}.
For example, when  $X$ is $S^3\times S^5(\simeq_{\Q}SU(3))$,
 $Baut_1X$ is not a rational H-space 
and  it is not even formal (see Definition \ref{formal} below).
Then the rational  cohomology  is infinitely generated  as
$$H^*(Baut_1X;\Q )\cong  \Q [v]\otimes \Lambda (w_0,w_1,w_2,\cdots )/(\{ v w_i\}_i ,\ \{ w_iw_j\}_{i<j} )$$
with $|v|(={\rm deg}(v))=4$ and  $|w_i|=3+6i $ even
though $H^*(Baut_1S^3;\Q )$ and $H^*(Baut_1S^5;\Q )$
are given as $\Q [v]$ and $\Q [u]$ with $|u|=6$, respectively.
A simply connected CW complex $X$ is said to be {\it elliptic} if 
the dimensions of rational cohomology and homotopy are  finite.
Futhermore, an elliptic space $X$ is said to be an {\it $F_0$-space} if 
$$ H^*(X;\Q )\cong  \Q [x_1,\cdots ,x_n]/(f_1,\cdots ,f_n),$$
in which $|x_i|$ are even and $f_1,\cdots ,f_n$ forms  a regular sequence in $ \Q [x_1,\cdots ,x_n]$.
Then the Sullivan model is given as 
$M(X)=(\Lambda (x_1,..,x_n,y_1,..,y_n),d)$ with $|x_i|$ even
and $|y_i|$ odd with $dx_i=0$ and $dy_i=f_i\in \Q [x_1,..,x_n]$.
In 1977,
S.Halperin \cite{Hal} conjectured that the Serre spectral sequences of all fibration $X\to E\to B$ collapse
for any $F_0$-space $X$.
The Halperin's conjecture for an $F_0$-space $X$ is equivalent to
$$ H^*(Baut_1X;\Q )\cong  \Q [v_1,\cdots ,v_m]$$
for some even-degrees' elements $v_1,\cdots ,v_m$ \cite{M}.
It is true when $X$ is a  homogeneous space $G/H$
where $G$ and $H$ be  compact connected Lie groups  with 
$H$ is a subgroup of $G$ and ${\rm rank}G={\rm rank}H$ due to Shiga-Tezuka \cite{ST}.
Of course, even if a space $X$ is not an $F_0$-space, 
 $ H^*(Baut_1X;\Q )$ may be a polynomial algebra.
For example, when $X$ is $S^3\times S^3$,
 $ H^*(Baut_1X;\Q )\cong  \Q[v,u]$
with $|v|=|u|=4$. (See \cite{LS} and \cite{Y} for the realization problem \cite[p.519]{FHT} of a polynomial algebra.)
 It is natural to ask
\begin{ques}\label{freeness}
When is $ H^*(Baut_1X;\Q )$ a polynomial algebra or free ?
\end{ques}

The meaning of ``$ H^*(Baut_1X;\Q )$ is free''
is that there is a graded algebra isomorphism $ H^*(Baut_1X;\Q )\cong \Lambda (x_1,..,x_m)\otimes \Q [y_{m+1},..,y_n]$ with $|x_i|(=|y_i|)=k_i$ for some $m\leq n$ (i.e., the differential of Sullivan mimimal model is zero). Then 
the rational homotopy set of classifying maps from a space $Y$ is given as 
$$[Y, (Baut_1X)_{\Q}]=[Y, \Pi_{i=1}^nK(\Q , k_i)]=\bigoplus_{i=1}^n H^{k_i}(Y;\Q ),$$
where $K(\Q , k_i)$ is the Eilenberg-Mac Lane space.
That is, the tuple of an n-elements of the cohomology of $Y$ rationally determine a fibration with fibre $X$ and  base $Y$.
Here $Z_{\Q}$ means the rationalization of a nilpotent space $Z$. 

In this paper, first, as an important obstruction for Problem \ref{freeness},
we pay attention to  
that $H^*(Baut_1(X\times Y);\Q )$ can  not be formulated by  $H^*(Baut_1X;\Q )$ and $H^*(Baut_1Y;\Q )$
since   $Baut_1X$ is not even a rational factor of $Baut_1(X\times Y)$, i.e., $(Baut_1(X\times Y))_{\Q}\not\simeq
(Baut_1X)_{\Q}\times C$ for any rational space $C$, 
 in general.
 Refer S.B.Smith's works  \cite{s2} and \cite{s1} in 2001.
 Thus our motivation is the following
 
\begin{lem} 
If $H^*(Baut_1(X\times Y);\Q )$ is free, then
$H^*(Baut_1X;\Q )$ and $H^*(Baut_1Y;\Q )$ are both free. But the converse is not true in general.
\end{lem}

We prepare the following general  notation to start our observation.   

\begin{defi}
A space $S$ with a map $f:S\to T$ is said to have  a rational retraction for $T$ when there is a map $r:T_{\Q}\to S_{\Q}$
with $r\circ f_{\Q}\simeq id_{S_{\Q}}$.
Also $S$ is said to have a weak (rational) retraction for $T$ when there is a graded Lie algebra map  $R:\pi_*(\Omega T)_{\Q}\to \pi_*(\Omega S)_{\Q}$
with $R\circ \pi_*(\Omega f)_{\Q}= id_{\pi_*(\Omega S)_{\Q}}$.
\end{defi}

Thus we have ``$\mbox{rational factor } \Rightarrow  \mbox{ rational retraction } \Rightarrow   \mbox{ weak retraction}$''.
In particular, when a map $f:S\to T$ is $\pi_{\Q}$-injective,
``$H^*(T;\Q)$ is free   $\Rightarrow$ $S$ is a rational factor of $T$''
from Lemma \ref{retr} below.
In general, a map $f:X\to Y$ does not induce a map between identity components of spaces 
of self-equivalences $aut_1X\to aut_1Y$.
Though the inclusion $i_X:X\to X\times Y$ induces
the map $j_X:aut_1X\to aut_1(X\times Y)$ with  $j_X(g)=g\times id_Y$,
where $id_Y$ is the identity of $Y$.
Thus there is the  morphism $Bj_X: Baut_1X\to Baut_1(X\times Y)$.
For example, $Baut_1S^3$ and  $Baut_1S^5$ are not  rational factors of $Baut_1(S^3\times S^5)$.
  $Baut_1S^3$ has a rational retraction for $Baut_1(S^3\times S^5)$ but $Baut_1S^5$ does  not even a weak retraction.
See \S 2 for the detail.
The following question  is a key to approach  Question \ref{freeness}.

\begin{ques}
When is $Baut_1X$ (or $Baut_1Y$)  a rational factor of $Baut_1(X\times Y)$  ?
More generally, when does  $Baut_1X$  (or $Baut_1Y$)  have  a rational (weak) retraction for  $Baut_1(X\times Y)$  ?
\end{ques}

 For example, recall  S.B.Smith's 
\begin{thm}(\cite[Theorem 3]{s2})\label{smith}
Suppose that  $X$ is an $F_0$-space with $Der_+H^*(X;\Q )=0$ and
$Y$ is a rational H-space.
If $\min \pi_*(X)_{\Q} +\min \pi_*(Y)_{\Q}\geq \max \pi_*(X\times Y)_{\Q}$
and $\max \pi_*(X)_{\Q} \leq \min \pi_*( Y)_{\Q}$, then 
$$Baut_1(X\times Y)\simeq_{\Q} Baut_1X\times ||Der(H^*(Y;\Q ),H^*(X\times Y;\Q ))||,$$
where $|| L||$ is the spatial  realization of a DGL $L$ \cite{Q}.
\end{thm}
Futhermore
there are some examples in \cite[section 5]{s2}.
In this paper, we consider a sufficient  condition for having a weak retraction in the case that $Y$ is an odd-sphere.


\begin{defi}\label{deriv}
(a) The $(m,k)$-evaluation  of $X$: $$\mu_{m,k}(X): H_m(Der M(X))\otimes H^k(X;\Q )\to H_{m-k}(Der M(X))$$ is given by 
$\mu_{m,k}([\sigma ]\otimes [w])=[\sigma\otimes w]$.
Here $([\sigma ]\otimes [w])(x):=[\sigma (x)w]$ for $x\in M(X)$.

(b) The $(m)$-cohomology induced map of $X$: 
$$\psi_m(X):H_m(DerM(X))\to Der_m H^*(X;\Q )$$
is given by  $\psi_m(X)([\sigma ])(w):=[\sigma (w)]$.
\end{defi}

\begin{prop}\label{main}
For the  odd-sphere $S^n$ $(n>1)$, 

(a) 
$Baut_1 X$ has a  weak retraction for $ Baut_1(X\times S^n)$
if 
$\mu_{m,k}(X)=0$ for any $0\leq k<n<m$.

(b) $Baut_1 S^n$ has a weak retraction for $Baut_1(X\times S^n)$ if 
$\psi_m(X)=0$
for any $0<m<n$. 

\end{prop}

\begin{defi}\label{formal}(\cite{T})
A simply connected space $X$ with Sullivan minimal model $M(X)=(\Lambda V,d)$ is said to be {\it formal}
if there is a quasi-isomorphism $M(X)\to (H^*(X;\Q ),0)$
and   $X$ is said to be {\it coformal} if 
there is a quasi-isomorphism from the Quillen's DGL-model \cite{Q} of $X$  to the rational homotopy Lie algebra 
 $\pi_*(\Omega X)_{\Q}$
(equivalently, the differential of Sullivan minimal model  is quadratic, i.e., 
$dV\subset \Lambda^2V$).
\end{defi}

When   $Baut_1(X\times S^n)$ is  coformal, 
``$Baut_1 X$ has a rational retraction for  $Baut_1(X\times S^n)$'' 
is equivalent to ``$Baut_1 X$ has a weak  retraction for  $Baut_1(X\times S^n)$''
from Lemma \ref{retr} below.

\begin{thm}\label{two}
Suppose that $H^*(Baut_1X;\Q )$ is  free.
Then the differential of $M(Baut_1(X\times S^n))$ does not have a quadratic part  for an odd-sphere $S^n$
if
(a)  $\mu_{m,k}(X)=0$
 for any $0\leq k<n<m$
and
(b) $\psi_m(X)=0$ 
for any $0<m<n$.

In particular, when  
$Baut_1(X\times S^n)$ is coformal,
 $H^*(Baut_1(X\times S^n);\Q )$ is  free if  (a) and (b) hold.
\end{thm}




An elliptic space $X$ is said to be {\it pure} if if $dV^{even}=0$ and $dV^{odd}\subset \Lambda V^{even}$
for $M(X)=(\Lambda V,d)$. For example,  simply connected Lie groups, 
homogeneous spaces  and $F_0$-spaces are   pure spaces.
More generally,  $X$ is said to be {\it two-stage} if if $dV_0=0$ and $dV_1\subset \Lambda V_0$
for $V=V_0\oplus V_1$.
From degree arguments, we get
\begin{exmp}
(1) $H^*(Baut_1X;\Q )$ is free
if and only if
$H^*(Baut_1(X\times S^2);\Q )$ is free.

(2) When 
$X$ is elliptic, $H^*(Baut_1(X\times S^3);\Q )$ is free
if and only if 
$M(X)\cong (\Lambda (x_1,..,x_m,y_1,..,y_n),d)$ with $|x_1|=\cdots =|x_m|=2$
and $|y_1|=\cdots =|y_n|=3$.


\end{exmp}

In the followings, the symbol $(v,f)$ means the {\it elementary derivation} that takes a generator $v$ of $V$ to an  element $f$ of $\Lambda V$
and the other generators to $0$.

\begin{exmp}\label{ex1}
Let $X$ be the pullback of the sphere bundle of the tangent bundle of $S^{a+b}$ by the canonical 
degree $1$ map $S^a\times S^b\to S^{a+b}$.
Then it is the tortal space of a fibration $S^{a+b-1}\to X\to S^a\times S^b$.
Notice that   $M(X)=(\Lambda (v_1,v_2,v_3),d)$
where $|v_1|=a$,  $|v_2|=b$,  $|v_3|=a+b-1$ are odd,
 $dv_1=dv_2=0$ and  $dv_3=v_1v_2$.
Then
$$H^*(X;\Q )\cong \frac{\Lambda (v_1,v_2)\otimes \Q [w_1,w_2]}{
(v_1v_2, v_1w_1,v_2w_2, v_1w_2+v_2w_1, w_1^2,w_2^2,w_1w_2     )}$$
with $w_1=[v_1v_3]$ and $w_2=[v_2v_3]$, so $X$ is not formal.
From \S 3(II),  $H^*(Baut_1X;\Q )$ is an one or two variable free algebra:
 $\Lambda (v_{3,0})$ or 
$\Lambda (v_{2,1},v_{3,0})$.
By referring to Theorem \ref{two},
  a   non-freeness condition of $H^*(Baut_1(X\times S^n);\Q)$ for an odd integer $n$
is given  by 
(a) $\mu_{|v_3|,0}([(v_3,1)],1)=[(v_3,1)]$,  $\mu_{|v_2|-|v_1|,0}([(v_2,v_1)],1)= [(v_2,v_1)]$ or 
(b) 
 $\psi_{|v_3|} ([(v_3,1)])=(w_1, v_1)+(w_2,v_2)$
under certain degree conditions.
From Theorem \ref{table} below, we have that 
$H^*(Baut_1(X\times S^n);\Q )$ is free  
if and only if $a=b$ and 
$2a-1\leq n<3a-1$.
 In \S 4, we give the Sullivan minimal models  for all cases of degrees.
 \end{exmp}




In this paper, second, we consider about Question  \ref{freeness} for
certain  pure spaces $X$.



\begin{thm}\label{cor}
Let $X$ be  a non-formal pure space where  $M(X)=(\Lambda (x_1,x_2,y_1,y_2 ,y_3),d)$ where $|x_i|$ is even
with $|x_1|\leq |x_2|$
and $|y_i|$ is odd with $|y_1|\leq |y_2| \leq  |y_3|$. Let $dy_i=f_i\in \Q [x_1,x_2]$ for $i=1,2,3$.
Then $H^*(Baut_1X;\Q )$ is not a polynomial algebra  if and only if 
the following (I) or (II) holds:\\
(I) There is an  odd-element $w\in H^*(X;\Q )$ with  $|w|<|y_3|$.\\
(II)   When $|x_1|<|x_2|$,
$$f_1=x_1^l\in\Q[x_1],\ \ \partial f_2/\partial x_2\cdot x_1^k\in (f_1),\ \   \partial f_3/\partial x_2\cdot x_1^k\in (f_1,f_2) \ \ \ \ \ \ \  \ \  (*)$$
for some $1<l$ and $0<k<\min \{ l,\ |x_2|/|x_1|\}$. 
Here  $\partial f_i/\partial x_j$ is the partial differentiation of $f_i$ by $x_j$ and
 $(S)$ is the ideal of $\Q [x_1,x_2]$ generated by a set $S$.
\end{thm}

In the following examples, (1) and (2) correspond to (I) and (II) of the above theorem, respectively.
Futhermore (3) satisfies the both cases of (I) and (II).
In all cases,  $Baut_1X$ are  not formal.

\begin{exmp}\label{su}

(1) When $X$ is the total space of a fibration 
$S^3\to X\to S^2\times \C P^3$ such that
$M(X)=(\Lambda (x_1,x_2,y_1,y_2,y_3),d)$
with 
 $|x_1|=|x_2|=2$, $|y_1|=|y_2|=3$, $|y_3|=7$, $dy_1=x_1^2$,   $dy_2=x_1x_2$,   $dy_3=x_2^4$.
There is the  even degree  non-zero element  $\tau =[(y_3,x_2y_1-x_1y_2)]\in H_2(DerM(X))$.
Then  we have 
$$M(Baut_1X)\cong (\Lambda (v_2,v_2',v_2'', v_3,v_4,v_4',v_4'', v_6,v_6', v_8),d)$$
with $|v_i|=|v_i'|=|v_i''|=i$, 
$dv_2=dv_2'=dv_2''=dv_3=dv_4=dv_4'=dv_8=0$,
$dv_4''=v_2v_3$,
$dv_6=v_3v_4$ and 
$dv_6'=v_3v_4'$.
Here 
the element $v_3$ corresponds to $\sigma$ (with respect to Theorem \ref{dermodel} below).

(2) When $X$ is the homogeneous space $SU(6)/SU(3)\times SU(3)$,
$M(X)=(\Lambda (x_1,x_2,y_1,y_2,y_3),d)$ with 
$|x_1|=4$, $|x_2|=6$, $|y_1|=7$, $|y_2|=9$, $|y_3|=11$, 
$dx_1=dx_2=0$, $dy_1=x_1^2$, $dy_2=x_1x_2$, $dy_3=x_2^2$ \cite{GHV}(\cite{A}).
There is the even degree   non-zero element  $\sigma=[(x_2,x_1)+(y_2,y_1)+2(y_3,y_2)]\in H_2(DerM(X))$.
Then we have 
$$M(Baut_1X)\cong (\Lambda (v_2,v_3,v_4,v_6,v_8,v_8', v_{10},v_{12}),d)$$
with $|v_i|=|v_i'|=i$, 
$dv_2=dv_3=dv_8=0$,
$dv_4=v_2v_3$,
$dv_6=v_3v_4$,
$dv_8'=v_3v_6$,
$dv_{10}=v_3v_8$ and 
$dv_{12}=v_{3}v_{10}$. Here 
the element $v_3$ corresponds to $\sigma$.

(3) When $X$ is the total space of a fibration 
$S^5\to X\to S^2\times \K P^2$ such that
$M(X)=(\Lambda (x_1,x_2,y_1,y_2,y_3),d)$
with 
 $|x_1|=2$, $|x_2|=4$, $|y_1|=3$, $|y_2|=5$, $|y_3|=11$, $dy_1=x_1^2$,   $dy_2=x_1x_2$,   $dy_3=x_2^3$.
There are the even degree  non-zero elements  
$\sigma=[(x_2,x_1)+(y_2,y_1)+3(y_3,x_2y_2)]\in H_2(DerM(X))$ and $\tau=[(y_3,x_2y_1-x_1y_2)]\in H_4(DerM(X))$.
Then  we have 
$$M(Baut_1X)\cong (\Lambda (v_2, v_3,v_4,v_4',v_5, v_6, v_8,v_{10},v_{12}),d)$$
with $|v_i|=|v_i'|=i$, 
$dv_2=dv_3=dv_4=dv_5=0$,
$dv_4'=v_2v_3$,
$dv_6=v_3v_4$, 
$dv_8=v_2v_3v'_4+v_4v_5+v_3v_6$,   
$dv_{10}=v_3v_8+v_5v_6$ and
$dv_{12}=v_3v'_4v_6$.
Here 
the elements $v_3$ and $v_5$ correspond to $\sigma$ and $\tau$, respectively.
It is not even coformal.



\end{exmp}


 Finally we give an example of a non-formal  homogeneous space $X$ that does not satisfy 
the conditions of Theorem \ref{cor}.

\begin{exmp}
\label{parent} 
Let $X$ be a homogeneous space  space with   $X=SU(2)\times SU(2)\times SU(2)/T^2$
of $T^2=S^1\times S^1$.
Due to \cite{Pa},
there are two types of  torus embeddings $T^2\subset SU(2)\times SU(2)\times SU(2)$
such that
$M(X)=(\Lambda (x_1,x_2,y_1,y_2,y_3),d)$
with $|x_1|=|x_2|=2$, $|y_1|=|y_2|=|y_3|=3$ and\\
(i) $dx_1=dx_2=0$, $dy_1=x_1^2$, $dy_2=x_2^2$,  $dy_3=0$ (formal)\\ 
(ii) $dx_1=dx_2=0$, $dy_1=x_1^2$, $dy_2=x_2^2$,  $dy_3=x_1x_2$ (non-formal).\\
In both cases, 
$$H^*(Baut_1X;\Q )\cong \Q [u_1,u_2,u_3,u_4,u_5,u_6,u_7]$$
with $|u_1|=|u_2|=|u_3|=|u_4|=2$ and $|u_5|=|u_6| =|u_7|=4$
since 
$\pi_{*+1}(Baut_1X)_{\Q}\cong H_*(DerM(X))\cong\Q \{ (y_1,x_2),(y_2,x_k),(y_3,x_1),(y_3,x_2),(y_1,1),(y_2,1),(y_3,1)\}$
 in which $k=1$ for (i) and  $k=2$ for (ii)
from Theorem \ref{dermodel}.
\end{exmp}

In \S 2, we prepare about computation of derivations  and give the proofs of Theorems  \ref{two} and \ref{cor}.
In \S 3, we compare $M(Baut_1X)$
for  formal cases and non-formal cases on elliptic spaces $X$ with rational homotopy groups of rank 3.
In \S 4, we give the Sullivan minimal models on all cases of Example \ref{ex1}.
In \S 5, we give the explicit computation for giving  the minimal model of Example \ref{su} (2)
but  omit the computations of Example \ref{su} (1) and (3) since they are  similar to it.


\section{Sullivan models and derivations}

We use the \textit{Sullivan minimal model} $M(X)$ 
 of a simply connected  space $X$ of finite type.
It is a free $\Q$-commutative differential graded algebra ({abbr.,} DGA) 
 $(\Lambda{V},d)$
 with a $\Q$-graded vector space $V=\bigoplus_{i> 1}V^i$
 where $\dim V^i<\infty$ and a decomposable differential,
 i.e., 
 $$\text{$d(V^i) \subset (\Lambda^+{V} \cdot \Lambda^+{V})^{i+1}$
 and\  $d \circ d=0.$}$$
Here  $\Lambda^+{V}$ is 
 the ideal of $\Lambda{V}$ generated by elements of positive degree.
 We often denote  $(\Lambda{V},d)$ simply 
 by  $\Lambda V$. 
{The degree of an element $x$ of a graded algebra is denoted by  $|{x}|$. }
Then  we have
$$\text{$xy=(-1)^{|{x}||{y}|}yx$ \ and\  $d(xy)=d(x)y+(-1)^{|{x}|}x \, d(y)$. }$$
We note that $M(X)$ determines the rational homotopy type $X_{\Q}$ of $X$.
In particular, there are  the following isomorphisms 
$$Hom(V^i,\Q )\cong \pi_i(X)\otimes {\Q} (=\pi_i(X)_{\Q})\mbox{  and  } H^*(\Lambda V,d)\cong H^*(X;\Q )$$
See \cite{FHT} for a general introduction and the standard notations.

Let $Der_i M(X)$ be the set of $\Q$-derivations of $M(X)$
decreasing the degree by $i$
with
$\sigma (xy)=\sigma (x)y+(-1)^{i\cdot |x|}x\sigma (y)$
for $x,y\in M(X)$. 
The boundary operator $\partial : Der_i M(X)\to Der_{i-1} M(X)$
is defined by $\partial  (\sigma)=d\circ \sigma-(-1)^i\sigma \circ d$
for $\sigma\in Der_iM(X)$.  
We denote  $\oplus_{i>0} Der_iM(X)$ by
$DerM(X)$
in which $Der_1M(X)$ is $\partial$-cycles.
Then ${Der}M(X)$ is a DGL by the Lie bracket $[\sigma ,\tau]:=\sigma\circ \tau-(-1)^{|\sigma||\tau|}
\tau\circ\sigma$.
 Furthermore, recall the definition (sign convention) of D.Tannr\'{e} \cite[p.25]{T}: 
$C^*(L,\partial )=(\Lambda s^{-1}\sharp L, D=d_1+d_2) $ with

$(i)\ \ \ \  \langle d_1s^{-1}z; sx\rangle =-\langle z;\partial x\rangle$ and

$(ii)\ \ \ \  \langle d_2s^{-1}z; sx_1,sx_2\rangle =(-1)^{|x_1|}\langle z;[x_1,x_2]\rangle$,\\
where $\langle s^{-1}z;sx\rangle=(-1)^{|z|}\langle z;x\rangle$ for a DGL $(L,\partial )$
and $\sharp L$ is the dual space of $L$.

\begin{thm}\label{dermodel}(\cite[p.314]{S}, \cite{FLS})
The DGA $C^*(DerM(X),\partial )$ is the Sullivan model $m(Baut_1X)$ of $Baut_1X$.
 In particular,  
$\pi_{*}(\Omega B{aut}_1 X)\otimes \Q\cong H_*({Der} M(X),\partial )$
as graded Lie algebras,
in which the left hand has the Samelson bracket.
\end{thm}

Here the DGA $m(Baut_1X)$ need  not be minimal (cf. Proposition  \ref{nonm} below).
The following lemma is obvious from DGA-arguments.

\begin{lem}\label{retr}
Suppose that  a map $f:S\to T$ is $\pi_*(f)_{\Q}$-injective. 
Then the model is given as  $M(f):M(T)=(\Lambda (U\oplus V),D)\to (\Lambda V,d)=M(S)$
with 
$$ M(f)\mid_V=id_V,\ M(f)(U)=0,\ d=\overline{D},\ DU\in \Lambda^+U\otimes \Lambda V,\   
DV\in \Lambda U\otimes \Lambda^+ V.$$
Then  $S$ is a rational factor of $T$ if and only if $DU\subset \Lambda U$ and $D\mid_V=d$.
Futhermore\\
(1) $S$ has a rational retraction for $T$  if and only if 
$DV\subset \Lambda V$ (i.e., $D\mid_V=d$).\\
(2)  $S$ has a weak  retraction for $T$ if and only if 
$DV\subset \Lambda V\oplus\Lambda^{>2}( U\oplus V)$.

\end{lem}

\begin{prop}\label{abc}Let $M(X\times S^n)=(\Lambda (V\oplus u),d)$ where $M(X)=(\Lambda V,d)$ and  $|u|=n$ for an odd-integer $n$.
Let $A=\oplus_{i>0} A_i$, $B=\oplus_{i>0} B_i$ and  $C=\oplus_{i>0} C_i$ where 
$$H_i(Der(\Lambda (V\oplus u))=A_i\oplus B_i\oplus C_i$$
$$:=H_i(Der(\Lambda V))\oplus  H_i(Der( \Lambda V,(u)))\oplus  H_i(Der(\Lambda  ( u),\Lambda V)),$$
where $(u)$ is the ideal of $\Lambda (V\oplus u)$ generated by $u$.
Then
\\
(1)  $Baut_1X$ has a weak retraction
for $Baut_1(X\times S^n)$ if  $[B,C]=0$.\\
(2) $Baut_1S^n$ has a weak retraction
for $Baut_1(X\times S^n)$ if  $[A,C]=0$.\\ 
(3) $\pi_*(\Omega Baut_1(X\times S^n))_{\Q}$ is abelian
(i.e., the Lie bracket of $H_*(Der(\Lambda (V\oplus u))$ is trivial ) if  
$[A,A]=0$, $[A,C]=0$ and  $[B,C]=0$.
\end{prop}

\begin{proof} In general, since they are negative degree derivations, we have
$$[A,A]=A, \ [A,B]=B,\ [A,C]=C, \ [B,B]=B,\  [B,C]=A,\  [C,C]=0.$$
From Theorem \ref{dermodel} and Lemma \ref{retr}, we have (1),(2) and (3). 
In particular, if $[A,A]=0$, we have $[A,B]=[B,B]=0$.
Thus  (3) holds.
\end{proof}

\begin{claim}\label{ABC}
(1) $\mu_{m,k}(X)=0$
 for any $0\leq k<n<m$
if and only if $[B,C]=0$.

(2) $\psi_m(X)=0$ 
for any $0<m<n$ if and only if $[A,C]=0$.
\end{claim}

\begin{proof}
 (1)  For $\sigma \in H_m(DerM(X))$
and $w\in H^k(X;\Q )$,
$$\mu_{m,k}(\sigma \otimes w)= \sigma\otimes w=
\pm [\sigma \otimes u,(u,w)]\in [B,C]$$
under $0\leq k<n<m$.
Conversely, all elements of $[B,C]$ are given as them.

(2) For  $\sigma \in H_m(DerM(X))$
and  $w\in H^k(X;\Q )$, 
$$(u,\psi_m(\sigma )(w))= (u,\sigma (w))=\pm [\sigma, (u,w)]\in [A,C]$$
under  $0<m\leq k<n$. Conversely, all elements of $[A,C]$ are given as them.
\end{proof}

\noindent
{\it Proof of Proposition \ref{main}. }
(a) 
 It holds from Proposition  \ref{abc} (1) and Claim  \ref{ABC} (1). 

(b) 
It holds  from Proposition  \ref{abc} (2) and Claim \ref{ABC} (2).\hfill\qed\\

\noindent
{\it Proof of Theorem \ref{two}.}
It holds from Proposition \ref{abc} (3) and Claim \ref{ABC} (1)(2).
\hfill\qed\\







From Theorem \ref{dermodel}, $\pi_{*+1}(B{aut}_1 Y)\otimes \Q\cong H_*({Der} M(Y))$.
Thus we have immediately by degree arguments
\begin{lem}\label{freepure}$H_{even}(DerM(X))=0$ if and only if 
 $H^*(Baut_1X;\Q )$ is a polynomial algebra.
\end{lem}

\noindent{\it Proof of Theorem \ref{cor}}.   
        From Lemma \ref{freepure},
it is sufficient to observe  when a non-zero  even degree element $\sigma$ exists in $H_*(DerM(X))$.
Then we can divide it two cases as (I) :``$|v|$ and $|f|$ are both odd'' or (II) :``$|v|$ and $|f|$ are both even''
for $\sigma=(v,f)+\cdots$.

 A possiblity of  $H_{even}(DerM(X))\neq 0$ is  that 
there is the non-exact $\delta$-cycle $\sigma=(y_3, w )$ for some odd-degree non-exact cocycle  $w$ in 
 the ideal $(y_1,y_2)$,
which is equivalent to (I).

The other possibility is given, when $|x_1|<|x_2|$,
as that there exists a non-exact $\delta$-cycle $\sigma=(x_2,x_1^k)+\tau$ for some 
$0<k<|x_2|/|x_1|$ and a derivation $\tau$.
Let $dy_i=f_i\in \Q [x_1,x_2]$ for $i=1,2,3$.
There exists such a  $\tau$ 
if and only if there are  elements  $g,h_1,h_2\in \Q[x_1,x_2]$ such that
  $$\partial  (x_2,x_1^k)=-(y_1,\partial f_1/\partial x_2\cdot x_1^k)-(y_2,\partial f_2/\partial x_2\cdot x_1^k)
-(y_3,\partial f_3/\partial x_2\cdot x_1^k)
$$
$$=(y_2,gf_1)+(y_3,h_1f_1+h_2f_2)=\partial  ((y_2,gy_1)+(y_3,h_1y_1+h_2y_2))=\partial (-\tau).$$
It is equivalent to   $(*)$ of (II).
\hfill\qed\\

\begin{rem}
In Theorem \ref{cor}(I), if  $\psi_{|y_i|} (y_i,1) (w)=[\partial w/\partial y_i]\neq 0$,
there is the non-exact $\delta$-cycle $\sigma=(y_3,\partial w/\partial y_i )$ with $\sigma =[(y_i,1),(y_3,w)]$
for  non-exact $\delta$-cycles $(y_i,1)$ and $(y_3,w)$ ($i=1$ or $2$).
Then the rational homotpy Lie algebra $\pi_*(\Omega Baut_1X)_{\Q}$ has  a non-trivial Lie bracket
(is not abelian) 
from Theoem \ref{dermodel}.
Especially $H^*(Baut_1X;{\Q})$
is not even free.
\end{rem}

\section{Oddly generated models of rank 3}
For elliptic spaces $X$ with 
$\dim\pi_*(X)_{\Q}= 3$ and $\pi_*(X)_{\Q}=\pi_{odd}(X)_{\Q}$, 
let $M(X)=(\Lambda (v_1,v_2,v_3),d)$  with $|v_i|$ odd for $i=1,2,3$
of $|v_1|\leq |v_2|\leq |v_3|$.
Then the minimal model $M:=M(Baut_1X)$ is given as follows:\\

 \noindent
 (I)\ \textit{ Formal case:  $d(v_1)=d(v_2)=d(v_3)=0$}.\\

(1.1) $|v_1|=|v_2|=|v_3|$
 $$M\cong (\Lambda (v_{1,0},v_{2,0},v_{3,0}),0)$$
     
(1.2) $|v_1|=|v_2|<|v_3|\leq 2|v_1|$
 $$M\cong (\Lambda (v_{1,0},v_{2,0},v_{3,0},v_{3,1},v_{3,2}),D)$$
  with 
  $D(v_{1,0})=D(v_{2,0})=D(v_{3,1})=D(v_{3,2})=0$
   and $D(v_{3,0})=v_{1,0}v_{3,1}+v_{2,0}v_{3,2}$.\\
  
(1.3) $|v_1|=|v_2|<\frac{1}{2}|v_3|$
 $$M\cong (\Lambda (v_{1,0},v_{2,0},v_{3,0},v_{3,1},v_{3,2},v_{3,12}),D)$$
  with   $D(v_{1,0})=D(v_{2,0})=D(v_{3,12})=0$,
  $D(v_{3,1})=v_{1,0}v_{3,12}$,
  $D(v_{3,2})=v_{2,0}v_{3,12}$ 
  and 
  $D(v_{3,0})=v_{1,0}v_{3,1}+v_{2,0}v_{3,2}$.\\

(1.4) $|v_1|<|v_2|<|v_3|\leq |v_1|+|v_2|$
 $$M\cong (\Lambda 
  (v_{1,0},v_{2,0},v_{2,1},v_{3,0},v_{3,1},v_{3,2}),D)$$
with  
  $D(v_{1,0})=D(v_{2,1})=0$,
$D(v_{2,0})=v_{1,0}v_{2,1}$,
  $D(v_{3,1})=v_{2,1}v_{3,2}$
  and 
  $D(v_{3,0})=v_{1,0}v_{3,1}+v_{2,0}v_{3,2}$.\\

(1.5) $|v_1|<|v_2|$ and $|v_1|+|v_2|<|v_3|$
 $$M\cong (\Lambda (v_{1,0},v_{2,0},v_{2,1},v_{3,0},v_{3,1},v_{3,2},v_{3,12}),D)$$
  with 
  $D(v_{1,0})=D(v_{2,1})=D(v_{3,2})=D(v_{3,12})=0$,
  $D(v_{2,0})=v_{1,0}v_{2,1}$,
  $D(v_{3,1})=v_{2,1}v_{3,2}+v_{3,12}v_{2,0}$
   and $D(v_{3,0})=v_{1,0}v_{3,1}+v_{2,0}v_{3,2}$.\\

(1.6) $|v_1|<|v_2|=|v_3|$
 $$M\cong (\Lambda (v_{1,0},v_{2,0},v_{2,1},v_{3,0},v_{3,1}),D)$$
  with 
  $D(v_{1,0})=D(v_{2,1})=D(v_{3,1})=0$,
  $D(v_{2,0})=v_{1,0}v_{2,1}$
     and $D(v_{3,0})=v_{1,0}v_{3,1}$.\\
    
 \noindent
(II)\ \textit{Non-formal case:
$d(v_1)=d(v_2)=0$ and $d(v_3)=v_1v_2$}.\\
 
 (2.1) 
 $|v_1|=|v_2|$
  $$M\cong (\Lambda (v_{3,0}),0).$$
  
(2.2) $|v_1|<|v_2|$
 $$M\cong (\Lambda (v_{2,1},v_{3,0}),0).$$

Here the elements 
$v_{i,j}$,   $v_{i,jk}$ and   $v_{i,0}$
corresponds to the derivations
$(v_i,v_j)$, $(v_i,v_jv_k)$ and   $(v_i,1)=v_i^*$,
respectively 
with respect to Theorem \ref{dermodel}.  
Thus $|v_{i,j}|=|v_i|-|v_j|+1$
is odd and  $|v_{i,jk}|=|v_i|-|v_j|-|v_k|+1$,  $|v_{i,0}|=|v_i|+1$ are even.
We also adopt these symbols to the following  sections again.

 \section{Example \ref{ex1}}

Let's directly  compute the rational homotopies  of $Baut_1(X\times S^n)$ 
when $M(X)=(\Lambda (v_1,v_2,v_3),d)$ with $d(v_1)=d(v_2)=0$ and $d(v_3)=v_1v_2$ (i.e.,  (II) of \S 3)
 and $H^*(S^n;\Q )=\Lambda (u)$ with $|u|=n$ odd.
The following cases (1) $\sim$ (6) exist:\\

(1) $|u|\leq |v_1|$\\

\ \ \ (i)  $|u|+|v_1|<|v_2|$
 $$M(Baut_1(X\times S^n) )\cong (\Lambda (v_{u,0},v_{2,1},v_{2,u1}, v_{3,0},v_{3,u}),d)$$
with $d(v_{3,0})=v_{u,0}v_{3,u}$ and 
 $d(v_{2,1})=v_{u,0}v_{2,u1}$ ($dv=0$ for the other generators $v$).\\

\ \ \ (ii)  $|u|+|v_1|>|v_2|$ and $|v_1|<|v_2|$
$$M(Baut_1(X\times S^n) )\cong (\Lambda (v_{u,0},v_{2,1},v_{3,0},v_{3,u}),d)$$
with
$d(v_{3,0})=v_{u,0}v_{3,u}$.\\

\ \ \ (iii)   $|v_1|=|v_2|$
$$M(Baut_1(X\times S^n) )\cong (\Lambda (v_{u,0},v_{3,0},v_{3,u}),d)$$
with $d(v_{3,0})=v_{u,0}v_{3,u}$.\\

(2)  $|v_1|<|u|\leq |v_2|$\\

\ \ \ (i)  $|u|+|v_1|<|v_2|$
$$M(Baut_1(X\times S^n))\cong (\Lambda (v_{u,0},v_{u,1}, v_{2,1},v_{2,u1}, v_{3,0},v_{3,u}),d)$$
with $d(v_{3,0})=v_{u,0}v_{3,u}$
and $d(v_{2,1})=v_{u,0}v_{2,u1}$.\\

\ \ \ (ii)  $|u|+|v_1|>|v_2|$
$$M(Baut_1(X\times S^n) )\cong (\Lambda (v_{u,0},v_{u,1}, v_{2,1}, v_{3,0},v_{3,u}),d)$$
with $d(v_{3,0})=v_{u,0}v_{3,u}$.\\

(3)  $|v_1|<|v_2|<|u|< |v_3|$
 $$M(Baut_1(X\times S^n))\cong (\Lambda (v_{2,1}, v_{u,0},v_{u,1}, v_{u,2}, v_{3,0},v_{3,u}),d)$$
 with  $d(v_{3,0})=v_{u,0}v_{3,u}$,  $d(v_{u,0})=v_{u,1}v_{u,2}v_{3,u}$ and
  $d(v_{u,1})=v_{2,1}v_{u,2}$.\\

(4)  $|v_1|=|v_2|<|u|< |v_3|$
 $$M(Baut_1(X\times S^n))\cong (\Lambda ( v_{u,0},v_{u,1}, v_{u,2}, v_{3,0},v_{3,u}),d)$$
with
 $d(v_{3,0})=v_{u,0}v_{3,u}$
 and $d(v_{u,0})=v_{u,1}v_{u,2}v_{3,u}$.\\

 (5)  $|v_1|<|v_2|<|v_3|\leq  |u|$\\

\ \ \ (i)  $|v_1|+|v_2|+|v_3|< |u|$
 $$M(Baut_1(X\times S^n))\cong (\Lambda ( v_{2,1},v_{3,0}, v_{u,0}, v_{u,1},v_{u,2},v_{u,13},v_{u,23},v_{u,123}),d)$$
with   $d(v_{u,0})=v_{3,0}^2v_{u,123}$, $d(v_{u,1})=v_{2,1}v_{u,2}+v_{3,0}v_{u,13}$,
 $d(v_{u,2})=v_{3,0}v_{u,23}$ and  $d(v_{u,13})=v_{2,1}v_{u,23}$.\\

 \ \ \ (ii)  $|v_2|+|v_3|<|u|\leq |v_1|+|v_2|+|v_3|$
 $$M(Baut_1(X\times S^n))\cong (\Lambda ( v_{2,1},v_{3,0}, v_{u,0}, v_{u,1},v_{u,2},v_{u,13},v_{u,23}),d)$$
 with $d(v_{u,1})=v_{2,1}v_{u,2}+v_{3,0}v_{u,13}$,
 $d(v_{u,2})=v_{3,0}v_{u,23}$ and 
 $d(v_{u,13})=v_{2,1}v_{u,23}$.\\

 \ \ \ (iii)  $|v_1|+|v_3|<|u|< |v_2|+|v_3|$
 $$M(Baut_1(X\times S^n))\cong (\Lambda ( v_{2,1},v_{3,0}, v_{u,0}, v_{u,1},v_{u,2},v_{u,13}),d)$$
 with 
 $d(v_{u,1})=v_{2,1}v_{u,2}+v_{3,0}v_{u,13}$.\\

 \ \ \ (iv)  $|u|< |v_1|+|v_3|$
 $$M(Baut_1(X\times S^n))\cong (\Lambda ( v_{2,1},v_{3,0}, v_{u,0}, v_{u,1},v_{u,2}),d)$$
 with
 $d(v_{u,1})=v_{2,1}v_{u,2}$.\\

 (6)  $|v_1|=|v_2|<|v_3|\leq  |u|$\\

\ \ \ (i)  $|v_1|+|v_2|+|v_3|< |u|$
 $$M(Baut_1(X\times S^n))\cong (\Lambda ( v_{3,0}, v_{u,0}, v_{u,1},v_{u,2},v_{u,13},v_{u,23},v_{u,123}),d)$$
 with 
 $d(v_{u,1})=v_{3,0}v_{u,13}$ and  $d(v_{u,2})=v_{3,0}v_{u,23}$.\\

 \ \ \ (ii)  $|v_2|+|v_3|<|u|\leq |v_1|+|v_2|+|v_3|$
 $$M(Baut_1(X\times S^n))\cong (\Lambda ( v_{3,0}, v_{u,0}, v_{u,1},v_{u,2},v_{u,13},v_{u,23}),d)$$
 with
 $d(v_{u,1})=v_{3,0}v_{u,13}$
 and $d(v_{u,2})=v_{3,0}v_{u,23}$.\\

 \ \ \ (iii)  $|u|< |v_1|+|v_3|$
$$M(Baut_1(X\times S^n))\cong (\Lambda ( v_{3,0}, v_{u,0}, v_{u,1},v_{u,2}),0).$$

Here 
$v_{u,0}$,  $v_{u,i}$, $v_{i,u}$, $\cdots$
correspond to the derivations 
$(u,1)$, $(u,v_i)$, $(v_i,u)$, $\cdots$ respectively. 
Thus $|v_{u,0}|$, $|v_{u,ij}|$, $|v_{i,uj}|$ are even and    
$|v_{i,u}|$, $|v_{u,i}|$, $|v_{u,ijk}|$ are odd. \\

 Thus we have the following

\begin{thm}\label{table}
When  $M(X)=(\Lambda (v_1,v_2,v_3),d)$
with $|v_i|$ odd,  $dv_1=dv_2=0$ and  $dv_3=v_1v_2$,
 there are 14-types of $Baut_1(X\times S^n)$ with $n$ odd as
{\small {\rm
\begin{center}
\begin{tabular}{|l|c|c|c|c|c|c|c|c|c|c|}
\hline
type & $f_X$&$r_X$&$w.r_X$&$f_{S^n}$&$r_{S^n}$&$w.r_{S^n}$&formal&coformal & rank&$H^*$-free\\
\hline
$(1)_i$ &no&no&no&no&yes&yes&no&yes&5&no\\
\hline
$(1)_{ii}$ &no&no&no&no&yes&yes&no&yes&4&no\\
\hline
$(1)_{iii}$ &no&no&no&no&yes&yes&no&yes&3&no\\
\hline
$(2)_{i}$ &no&no&no&no&yes&yes&no&yes&6&no\\
\hline
$(2)_{ii}$ &no&no&no&no&yes&yes&no&yes&5&no\\
\hline
$(3)$ &no&no&no&no&no&yes&no&no&6&no\\
\hline
$(4)$ &no&no&no&no&no&yes&no&no&5&no\\
\hline
$(5)_{i}$&no&yes&yes&no&no&yes&no&no&8&no\\
\hline
$(5)_{ii}$&no&yes&yes&yes&yes&yes&no&yes&7&no\\
\hline
$(5)_{iii}$&no&yes&yes&yes&yes&yes&yes&yes&6&no\\
\hline
$(5)_{iv}$&no&yes&yes&yes&yes&yes&no&yes&5&no\\
\hline
$(6)_{i}$&no&yes&yes&yes&yes&yes&yes&yes&7&no\\
\hline
$(6)_{ii}$&no&yes&yes&yes&yes&yes&yes&yes&6&no\\
\hline
$(6)_{iii}$&yes&yes&yes&yes&yes&yes&yes&yes&4&yes\\
\hline
\end{tabular}
\end{center}
}}
\noindent
Here $f_X$ means that $Baut_1X$ is a rational factor of $Baut_1(X\times S^n)$,
$r_X$ means that $Baut_1X$ has  a rational retraction for $Baut_1(X\times S^n)$ and 
$w.r_X$ means that $Baut_1X$ has a weak rectraction  for $Baut_1(X\times S^n)$.
Futhermore $f_{S^n}$, $r_{S^n}$ and $w.r_{S^n}$ are similar symbols for $Baut_1S^n$.
$H^*$-free means that $H^*(Baut_1(X\times S^n);\Q)$ is free.
\end{thm}

\begin{proof}
Since 
$M(Baut_1X)=(\Lambda v_{3,0},0)$ or $(\Lambda (v_{3,0},v_{2,1}),0)$
and $M(Baut_1S^n)=(\Lambda v_{u,0},0)$,
we can check  whether or not  the properties $f_X$, .., $w.r_{S^n}$ are true respectively from Lemma  \ref{retr}.
See \cite[\S 4]{s2} for   formality conditions.
\end{proof}



\begin{cor} Suppose  $M(X)=(\Lambda (v_1,v_2,v_3),d)$ where $|v_i|$ are odd, $dv_1=dv_2=0$,   $dv_3=v_1v_2$
and 
$n$ is odd. Then\\
(1) The followings are equivalent:

(i) $H^*(Baut_1(X\times S^n);\Q )$ is free,

(ii) 
$|v_1|=|v_2|<|v_3|\leq n<|v_1|+|v_3|$ and

(iii)  $Baut_1X$ is a rational factor of $Baut_1(X\times S^n)$.\\
(2) $Baut_1(X\times S^n)$ is coformal if and only if $Baut_1S^n$ has  a rational retraction for $Baut_1(X\times S^n)$.\\
(3) If $Baut_1(X\times S^n)$ is formal,  $Baut_1S^n$ is a rational factor of $Baut_1(X\times S^n)$
and $Baut_1X$ has  a rational retraction for $Baut_1(X\times S^n)$.\\
(4) $Baut_1X$ has  a rational retraction for $Baut_1(X\times S^n)$
if and only if it does a weak retraction.\\
(5) $Baut_1S^n$ has  a weak  retraction for $Baut_1(X\times S^n)$.
\end{cor}




\section{Case of $X={SU(6)}/{SU(3)\times SU(3)}$}

Let 
$X={SU(6)}/{SU(3)\times SU(3)}$.
Then 
$M(X)=(\Lambda (x_1, x_2, y_1, y_2, y_3),  d) $
with 
$ |x_1  |=4$, $|x_2  |=6$, $|y_1  |=7$,   $|y_2  |=9$, $|y_3  |=11$,
$dx_1=dx_2=0$, $dy_1=x_1^2$, $dy_2=x_1x_2$ and  $dy_3=x_2^2$. 
Then a  basis of $DerM(X)$  is given as the table:
\begin{center}
\begin{tabular}{l|l}

\multicolumn{1}{c|}{degree} &
\multicolumn{1}{c}{generators}  \\
      
\hline
$11$ & $(y_3,1)$  \\
\hline
$9$ & $(y_2,1)$ \\
\hline
$7$ & $(y_3,x_1)$, \ $(y_1,1)$ \\
\hline
$6$ & $(x_2,1)$ \\
\hline
$5$ & $(y_3,x_2)$, \ $(y_2,x_1)$ \\
\hline
$4$ & $(y_3,x_1)$, \ $(x_1,1)$ \\
\hline
$3$ & $(y_3,x_1^2)$, \ $(y_2,x_2)$, \ $(y_1,x_1)$ \\
\hline
$2$ & $(y_3,x_2)$, \ $(y_2,x_1)$, $(x_2,x_1)$ \\
\hline
$1$ & $(y_3,x_1x_2)$, \ $(y_2,x_1^2)$, \ $(y_1,x_2)$ \\
\end{tabular}\\
\end{center}
\noindent

\begin{prop}For $M(X)=(\Lambda (x_1, x_2, y_1, y_2, y_3),  d) $ of above, we have 
$$\pi_{*+1}(Baut_1X)\otimes \Q     \cong
\Q \{ \ s(y_3,1), \ s(y_2,1), \ s(y_1,1), \ s(y_1,x_1), \ s(y_1,x_2),\ s(y_3,x_1), \ s(y_3,x_2),   $$
$$ s(2(y_3,y_2)+(y_2,y_1)+(x_2,x_1)) \ \}$$
as a graded vector space.
Here $|s((v,f)+\cdots) |=|v|-|f|+1$.
\end{prop}

\begin{proof}
Recall \S 2. The differential $\partial$ on the generators of of $DerM(X)$ is given as 
\begin{align*}
&\partial \left((y_3,1) \right )=\partial \left((y_2,1) \right )=\partial \left((y_3,x_1) \right )=\partial \left((y_1,1) \right )=\partial \left((y_3,x_2) \right )= \partial \left((y_2,x_1) \right )=0 \\
&\partial \left((y_3,x_1x_2) \right )= \partial \left((y_2,x_1^2) \right )= \partial \left((y_1,x_2) \right )=\partial \left((y_3,x_1^2) \right )= \partial \left((y_2,x_2) \right )= \partial \left((y_1,x_1) \right )=0\\
&\partial \left((x_2,1) \right )=-(y_2,x_1)-2(y_3,x_2),\  \partial \left((y_3,y_1) \right )=(y_3,x_1^2), \ \partial \left((x_1,1) \right )=-2(y_1,x_1)-(y_2,x_2) \\
&\partial \left((y_3,y_2) \right )=(y_3,x_1x_2), \ \partial \left((y_2,y_1) \right )=(y_2,x_1^2), \ \partial \left((x_2,x_1) \right )=-(y_2,x_1^2)-2(y_3,x_1x_2). 
\end{align*}
Then we obtain the result from $\pi_{*+1}(Baut_1X)\otimes \Q\cong H_*(DerM(X))=Ker(\partial)/Im(\partial)$ of Theorem \ref{dermodel}.
\end{proof}

\begin{prop}\label{nonm}
The DGA-model $m(Baut _1X):=C^*(DerM(X))$ is given as 
\begin{align*}
m(Baut _1X)\cong (\Lambda (&V_{y_3,1}, V_{y_2,1}, V_{y_3,x_1}, V_{y_1,1}, V_{x_2,1}, V_{w_1}, V_{y_2,x_1}, V_{y_3,y_1}, V_{x_1,1}, \\
&V_{y_3,x_1^2}, V_{y_2,x_2}, V_{w_2}, V_{y_3,y_2}, V_{y_2,y_1}, V_{x_2,x_1}, V_{y_3,x_1x_2}, V_{y_2,x_1^2}, V_{y_1,x_2}), D  ) 
\end{align*}
\noindent
with
$|V_{y_3,1}  |=12$, $|V_{y_2,1} |=10$, $ |V_{y_3,x_1}  |= |V_{y_1,1}  |=8$,  $ |V_{x_2,1}  |=7$, 
$ |V_{w_1}  |= |V_{y_2,x_1}  |=6$, $ |V_{y_3,y_1}  |= |V_{x_1,1}  |=5$, $ |V_{y_3,x_1^2}  |= |V_{y_2,x_2}  |= |V_{w_2}  |=4$, 
$ |V_{y_3,y_2}  |= |V_{y_2,y_1}  |= |V_{x_2,x_1}  |=3$, $ |V_{y_3,x_1x_2}  |= |V_{y_2,x_1^2}  |= |V_{y_1,x_2}  |=2$ and differential:
\begin{align*}
&D(V_{y_3,1})=-V_{x_1,1}V_{y_3,x_1}+V_{x_2,1}V_{w_1}-2V_{x_2,1}V_{y_2,x_1}+V_{y_1,1}V_{y_3,y_1}+V_{y_2,1}V_{y_3,y_2} \\
&D(V_{y_2,1})=-V_{x_1,1}V_{y_2,x_1}-V_{x_2,1}V_{y_2,x_2}+V_{y_1,1}V_{y_2,y_1} \\
&D(V_{y_3,x_1})=V_{x_2,x_1}V_{w_1}-2V_{x_2,x_1}V_{y_2,x_1}-V_{w_2}V_{y_3,y_1}+2V_{y_2,x_2}V_{y_3,y_1}-2V_{x_1,1}V_{y_3,x_1^2}\\
&\ \ \ \ \ \ \ \ \ \ \ \ \ +V_{y_2,x_1}V_{y_3,y_2}-V_{x_2,1}V_{y_3,x_1x_2} \\
&D(V_{y_1,1})=V_{x_1,1}V_{w_2}-2V_{x_1,1}V_{y_2,x_2}-V_{x_2,1}V_{y_1,x_2} \\
&D(V_{x_2,1})=-V_{x_1,1}V_{x_2,x_1} \\
&D(V_{w_1})=-2V_{x_2,x_1}V_{y_2,x_2}-4V_{x_1,1}V_{y_2,x_1^2}+4V_{y_2,x_2}V_{y_2,y_1}-2V_{w_2}V_{y_2,y_1}-V_{y_1,x_2}V_{y_3,y_1}\\
&\ \ \ \ \ \ \ \ \ \ \ \ \ -V_{y_2,x_2}V_{y_3,y_2}+V_{x_1,1}V_{y_3,x_1x_2} \\
&D(V_{y_2,x_1})=V_{x_2,1}-V_{x_2,x_1}V_{y_2,x_2}-V_{w_2}V_{y_2,y_1}+2V_{y_2,x_2}V_{y_2,y_1}-2V_{x_1,1}V_{y_2,x_1^2} \\
&D(V_{y_3,y_1})=-V_{y_2,y_1}V_{y_3,y_2} \\
&D(V_{y_3,x_1^2})=-V_{y_3,y_1}-V_{x_2,x_1}V_{y_3,x_1x_2}+V_{y_2,x_1^2}V_{y_3,y_2} \\
&D(V_{y_2,x_2})=V_{x_1,1}+V_{y_1,x_2}V_{y_2,y_1} \\
&D(V_{w_2})=2V_{y_1,x_2}V_{y_2,y_1}+V_{x_2,x_1}V_{y_1,x_2} \\
&D(V_{y_3,x_1x_2})=-V_{y_3,y_2}+2V_{x_2,x_1} \\
&D(V_{y_2,x_1^2})=-V_{y_2,y_1}+V_{x_2,x_1} \mbox{ and}\\
&D(V_{x_1,1})=D(V_{y_3,y_2})=D(V_{y_2,y_1})=D(V_{x_2,x_1})=D(V_{y_1,x_2})=0. \\
\end{align*}
\end{prop}

\begin{proof}
Recall \S 2 $(i)$.
The differential $d_1$ is given as 
$$d_1(V_{y_3,1})=d_1(V_{y_2,1})=d_1(V_{y_3,x_1})=d_1(V_{y_1,1})=d_1(V_{x_2,1})=d_1(V_{y_3,y_1})=0,$$
$$d_1(V_{y_3,y_2})=d_1(V_{y_2,y_1})=d_1(V_{x_2,x_1})=d_1(V_{x_1,1})=d_1(V_{y_1,x_2})=0, $$
$$d_1(V_{y_3,x_2})=2V_{x_2,1}, \ d_1(V_{y_2,x_1})=V_{x_2,1},\ d_1(V_{y_3,x_1^2})=-V_{y_3,y_1},\  d_1(V_{y_2,x_2})=V_{x_1,1},$$
$$\ d_1(V_{y_1,x_1})=2V_{x_1,1}, \ d_1(V_{y_3,x_1x_2})=-V_{y_3,y_2}+2v_{x_2,x_1}, \ d_1(V_{y_2,x_1^2})=-V_{y_2,y_1}+V_{x_2,x_1} .$$

Next, the Lie bracket is given as 
$$[(x_1,1), (y_3,x_1)]=(y_3,1), \
[(x_2,1), (y_3,x_2)]=(y_3,1), \
[(y_1,1), (y_3,y_1)]=(y_3,1), $$
$$[(y_2,1), (y_3,y_2)]=(y_3,1), \
[(x_1,1), (y_2,x_1)]=(y_2,1), \
[(x_2,1), (y_2,x_2)]=(y_2,1), $$
$$[(y_1,1), (y_2,y_1)]=(y_2,1), \
[(x_2,x_1), (y_3,x_2)]=(y_3,x_1), \
[(y_1,x_1), (y_3,y_1)]=(y_3,x_1), $$
$$\frac{1}{2}[(x_1,1), (y_3,x_1^2)]=(y_3,x_1), \
[(y_2,x_1), (y_3,y_2)]=(y_3,x_1), \
[(x_2,1), (y_3,x_1x_2)]=(y_3,x_1), $$
$$[(x_1,1), (y_1,x_1)]=(y_1,1), \
[(x_2,1), (y_1,x_2)]=(y_1,1), \
[(x_1,1), (x_2,x_1)]=(x_2,1), $$
$$[(y_1,x_2), (y_3,y_1)]=(y_3,x_2), \
[(y_2,x_2), (y_3,y_2)]=(y_3,x_2), \
[(x_1,1), (y_3,x_1x_2)]=(y_3,x_2), $$
$$[(x_2,x_1), (y_2,x_2)]=(y_2,x_1), \
[(y_1,x_1), (y_2,y_1)]=(y_2,x_1), \
\frac{1}{2}[(x_1,1), (y_2,x_1^2)]=(y_2,x_1), $$
$$[(y_2,y_1), (y_3,y_2)]=(y_3,y_1), \
[(x_2,x_1), (y_3,x_1x_2)]=(y_3,x_1^2), \
[(y_2,x_1^2), (y_3,y_2)]=(y_3,x_1^2), $$
$$[(y_1,x_2), (y_2,y_1)]=(y_2,x_2) \
\mbox{ and } \ [(x_2,x_1), (y_1,x_2)]=(y_1,x_1). \ \ \ \ \ \ \ \ \ \ \ \ \ \ \ \ \ \ \ \ \ \ \ \ 
$$

Then the  differential 
$d_2$ of \S 2 $(ii)$ is given as 
\begin{align*}
&d_2(V_{y_3,1})=-V_{x_1,1}V_{y_3,x_1}-V_{x_2,1}V_{y_3,x_2}+V_{y_1,1}V_{y_3,y_1}+V_{y_2,1}V_{y_3,y_2} \\
&d_2(V_{y_2,1})=-V_{x_1,1}V_{y_2,x_1}-V_{x_2,1}V_{y_2,x_2}+V_{y_1,1}V_{y_2,y_1} \\
&d_2(V_{y_3,x_1})=-V_{x_2,x_1}V_{y_3,x_2}+V_{y_1,x_1}V_{y_3,y_1}-2V_{x_1,1}V_{y_3,x_1^2}+V_{y_2,x_1}v_{y_3,y_2}-V_{x_2,1}V_{y_3,x_1x_2} \\
&d_2(V_{y_1,1})=-V_{x_1,1}V_{y_1,x_1}-V_{x_2,1}V_{y_1,x_2} \\
&d_2(V_{x_2,1})=-V_{x_2,1}V_{y_1,x_2} \\
&d_2(V_{y_3,x_2})=V_{y_1,x_2}V_{y_3,y_1}+V_{y_2,x_2}V_{y_3,y_2}-V_{x_1,1}V_{y_3,x_1x_2} \\
&d_2(V_{y_2,x_1})=-V_{x_2,x_1}V_{y_2,x_2}+V_{y_1,x_1}V_{y_2,y_1}-2V_{x_1,1}V_{y_2,x_1^2} \\
&d_2(V_{y_3,y_1})=-V_{y_2,y_1}V_{y_3,y_1} \\
&d_2(V_{y_3,x_1^2})=-V_{x_2,x_1}V_{y_3,x_1x_2}+V_{y_2,x_1^2}V_{y_3,y_2} \\
&d_2(V_{y_2,x_2})=V_{y_1,x_2}V_{y_2,y_1} \\
&d_2(V_{y_1,x_1})=-V_{x_2,x_1}V_{y_1,x_2} \ \ \  {\rm and} \\
&d_2(V_{x_1,1})=d_2(V_{y_3,y_2})=d_2(V_{y_2,y_1})=d_2(V_{x_2,x_1})=d_2(V_{y_3,x_1x_2})=d_2(V_{y_2,x_1^2})=d_2(V_{y_1,x_2})=0. 
\end{align*}
Let 
$V_{w_1}:=2V_{y_2,x_1}-V_{y_3,x_2}$ and  $V_{w_2}:=2V_{y_2,x_2}-V_{y_1,x_1} $.
Then, for $D=d_1+d_2$,  we have done from Theorem \ref{dermodel}.
\end{proof}

\begin{lem}\label{dga}Let $M=(\Lambda V,d)$ be a minimal DGA and $m=(\Lambda U,D)$ a free DGA with $D=d_1+d_2$.
Here $d_i:U\to \Lambda^i U$. 
Suppose that  there is a DGA-map $\varphi :M\to m$ such that $\varphi^{\sharp}:V\cong H^*( U, d_1)$.
Then  $\varphi $ is a quasi-isomorphism, i.e., $M$ is the minimal model of $m$.
\end{lem}

\begin{proof}There are  the  filtrations of the word lengths on $\Lambda V$ and $\Lambda U$, respectively:
$$F^i(\Lambda V):=\Lambda^{\geq i}V\mbox{ \ and \ }F^i(\Lambda U):= \Lambda^{\geq i}U,$$
where $F^0\supset \cdots \supset F^n\supset F^{n+1}\supset \cdots$
and $F^NH^m(\Lambda V)=F^NH^m(\Lambda U)=0$
for a sufficient large $N$ for each $m$.
Note that  differentials and a DGA-map preserve the filtrations.  
From the assumption,   $\varphi $ induces an isomorphism on $E^1$-term of
in the spectral sequence 
of the (decreasing) filtrations. Thus it indicates an isomorphism 
$\varphi^* :H^*(M)\cong  H^*(m)$
 from 
\cite[Theorem 3.2]{Mc}
(\cite[Chap.9 Theorem 3]{Sp}).
\end{proof}

\begin{thm}
The minimal model $M(Baut _1X)$ of the DGA-model $m(Baut_1X)$  is given as follows:
$$M(Baut _1X)\cong (\Lambda (U_{y_3,1}, U_{y_2,1}, U_{y_1,1}, U_{y_3,x_1}, U_{w_1}, U_{w_2}, U_{\sigma }, U_{y_1,x_2}), d )$$
where
$ |U_{y_3,1}  |=12$, $ |U_{y_2,1} |=10$, $ |U_{y_1,1}  |=8$,  $|U_{y_3,x_1}|=8$, 
$|U_{w_1}  |=6$, $ |U_{w_2}  |=4$, $ |U_{\sigma }  |=3$, $|U_{y_1,x_2}  |=2$
with differential
$d(U_{y_3,1})=U_{y_2,1}U_{\sigma }$, 
$d(U_{y_2,1})=U_{y_1,1}U_{\sigma }$, 
$d(U_{y_3,x_1})=U_{w_1}U_{\sigma }$, 
$d(U_{w_1})=U_{w_2}U_{\sigma } $,
$d(U_{w_2})=U_{y_1,x_2}U_{\sigma }$ and 
$d(U_{y_1,1})=d(U_{\sigma })=d(U_{y_1,x_2})=0 
$.
Here $U_{\sigma }$ is a non-exact cocycle corresponding to $\sigma =[(x_2,x_1)+(y_2,y_1)+2(y_3,y_2)]$.
It is not formal but coformal.
\end{thm}

\begin{proof}
Indeed, a DGA-map 
$\varphi \colon M(Baut_1X) \to m(Baut_1X)$ is given as 
{\small 
\begin{align*}
\varphi (U_{y_1,x_2})=&\ V_{y_1,x_2} \\
\varphi (U_{\sigma })=&\ V_{x_2,x_1}+2V_{y_3,y_2}+V_{y_2,y_1} \\
\varphi (U_{w_2})=&\ 2V_{w_2}+3V_{y_2,x_1^2}V_{y_1,x_2}-2V_{y_3,x_1x_2}V_{y_1,x_2} \\
\varphi (U_{w_1})=&\ -6V_{w_1}-4V_{w_2}V_{y_3,x_1x_2}+10V_{w_2}V_{y_2,x_1^2}+6V_{y_2,x_2}V_{y_3,x_1x_2} -24V_{y_2,x_2}V_{y_2,x_1^2} \\
&+6V_{y_3,x_1^2}V_{y_1,x_2}+2V_{y_3,x_1x_2}^2V_{y_1x_2}-\dfrac{7}{2}V_{y_2,x_1^2}^2V_{y_1,x_2} \\
\varphi (U_{y_3,x_1})=&\ -36V_{y_3,x_1}+12V_{w_1}V_{y_3,x_1x_2}+6V_{w_1}V_{y_2,x_1^2}-36V_{y_2,x_1}V_{y_3,x_1x_2}-72V_{y_2,x_2}V_{y_3,x_1^2} \\
&+36V_{w_2}V_{y_3,x_1^2}+4V_{w_2}V_{y_3,x_1x_2}^2-11V_{w_2}V_{y_2,x_1^2}^2+16V_{w_2}V_{y_3,x_1x_2}V_{y_2,x_1^2} \\
&-12V_{y_2,x_2}V_{y_3,x_1x_2}^2+24V_{y_2,x_2}V_{y_2,x_1^2}^2-30V_{y_2,x_2}V_{y_3,x_1x_2}V_{y_2,x_1^2} \\
&-12V_{y_3,x_1^2}V_{y_3,x_1x_2}V_{y_1,x_2}-6V_{y_3,x_1^2}V_{y_1,x_2}V_{y_2,x_1^2}-\dfrac{4}{3}V_{y_3,x_1x_2}^3V_{y_1,x_2} \\
&+\dfrac{11}{6}V_{y_2,x_1^2}^3V_{y_1,x_2}-6V_{y_3,x_1x_2}^2V_{y_1,x_2}V_{y_2,x_1^2}+V_{y_2,x_1^2}^2V_{y_3,x_1x_2}V_{y_1,x_2} \\
\varphi (U_{y_1,1})=&\ 3V_{y_1,1}+3V_{y_2,x_1}V_{y_1,x_2}+V_{w_2}^2+3V_{y_2,x_2}^2-3V_{w_2}V_{y_2,x_2} +6V_{y_2,x_2}V_{y_1,x_2}V_{y_2,x_1^2} \\
&-2V_{w_2}V_{y_1,x_2}V_{y_2,x_1^2}+V_{y_1,x_2}^2V_{y_2,x_1^2}^2 \\
\varphi (U_{y_2,1})=&\ 18V_{y_2,1}-6V_{y_1,1}V_{y_3,x_1x_2}+15V_{y_1,1}V_{y_2,x_1^2}+9V_{y_3,x_1}V_{y_1,x_2}-3V_{w_1}V_{w_2}+18V_{y_2,x_1}V_{y_2,x_2} \\
&-6V_{w_1}V_{y_1,x_2}V_{y_2,x_1^2}+3V_{y_2,x_1}V_{y_3,x_1x_2}V_{y_1,x_2}+15V_{y_2,x_1}V_{y_1,x_2}V_{y_2,x_1^2}-2V_{w_2}^2V_{y_3,x_1x_2} \\
&+5V_{w_2}^2V_{y_2,x_1^2}-6V_{y_2,x_2}^2V_{y_3,x_1x_2}+33V_{y_2,x_2}^2V_{y_2,x_1^2}+9V_{w_2}V_{y_2,x_2}V_{y_3,x_1x_2}-27V_{w_2}V_{y_2,x_2}V_{y_2,x_1^2} \\
&-6V_{w_2}V_{y_3,x_1^2}V_{y_1,x_2}+18V_{y_2,x_2}V_{y_3,x_1^2}V_{y_1,x_2}-4V_{w_2}V_{y_2,x_1^2}^2V_{y_1,x_2}-2V_{w_2}V_{y_3,x_1x_2}V_{y_1,x_2}V_{y_2,x_1^2} \\
&+6V_{y_2,x_2}V_{y_2,x_1^2}^2V_{y_1,x_2}+12V_{y_2,x_2}V_{y_3,x_1x_2}V_{y_1,x_2}V_{y_2,x_1^2}+6V_{y_3,x_1^2}V_{y_1,x_2}^2V_{y_2,x_1^2} \\
&+4V_{y_1,x_2}^2V_{y_2,x_1^2}^2V_{y_3,x_1x_2}-V_{y_1,x_2}^2V_{y_2,x_1^2}^3 
\end{align*}
\begin{align*}
\varphi (U_{y_3,1})=&\ 54V_{y_3,1}+18V_{y_2,1}V_{y_3,x_1x_2}-18V_{y_2,1}V_{y_2,x_1^2}+54V_{y_1,1}V_{y_3,x_1^2}+54V_{y_3,x_1}V_{y_2,x_2}+18V_{w_1}^2\\
&+54V_{y_2,x_1}^2 -54V_{w_1}V_{y_2,x_1}+6V_{y_1,1}V_{y_3,x_1x_2}^2-\frac{33}{2}V_{y_1,1}V_{y_2,x_1^2}^2+24V_{y_1,1}V_{y_3,x_1x_2}V_{y_2,x_1^2}\\
&-18V_{y_3,x_1}V_{y_3,x_1x_2}V_{y_1,x_2} +45V_{y_3,x_1}V_{y_1,x_2}V_{y_2,x_1^2}+6V_{w_1}V_{w_2}V_{y_3,x_1x_2}-15V_{w_1}V_{w_2}V_{y_2,x_1^2}\\
&-36V_{w_1}V_{y_2,x_2}V_{y_3,x_1x_2}+36V_{w_1}V_{y_2,x_2}V_{y_2,x_1^2} +6V_{y_2,x_2}^2V_{y_3,x_1x_2}V_{y_2,x_1^2}-9V_{y_2,x_2}V_{y_2,x_1^2}^3V_{y_1,x_2}\\
&-36V_{w_1}V_{y_3,x_1^2}V_{y_1,x_2}+72V_{y_2,x_1}V_{y_2,x_2}V_{y_3,x_1x_2}-18V_{y_2,x_1}V_{y_2,x_2}V_{y_2,x_1^2}+54V_{y_2,x_1}V_{y_3,x_1^2}V_{y_1,x_2} \\
&+108V_{y_2,x_2}^2V_{y_3,x_1^2}-54V_{w_2}V_{y_2,x_2}V_{y_3,x_1^2}+6V_{w_1}V_{y_2,x_1^2}^2V_{y_1,x_2}-24V_{w_1}V_{y_3,x_1x_2}V_{y_1,x_2}V_{y_2,x_1^2} \\
&-12V_{y_2,x_1}V_{y_3,x_1x_2}^2V_{y_1,x_2}-\frac{33}{2}V_{y_2,x_1}V_{y_2,x_1^2}^2V_{y_1,x_2}+69V_{y_2,x_1}V_{y_3,x_1x_2}V_{y_1,x_2}V_{y_2,x_1^2} +2V_{w_2}^2V_{y_3,x_1x_2}^2 \\
&+\frac{25}{2}V_{w_2}^2V_{y_2,x_1^2}^2-10V_{w_2}^2V_{y_3,x_1x_2}V_{y_2,x_1^2}+24V_{y_2,x_2}^2V_{y_3,x_1x_2}^2+\frac{75}{2}V_{y_2,x_2}^2V_{y_2,x_1^2}^2  \\
&+18V_{y_3,x_1^2}^2V_{y_1,x_2}^2-12V_{w_2}V_{y_2,x_2}V_{y_3,x_1x_2}^2-\frac{87}{2}V_{w_2}V_{y_2,x_2}V_{y_2,x_1^2}^2+15V_{w_2}V_{y_2,x_2}V_{y_3,x_1x_2}V_{y_2,x_1^2} \\
&+12V_{w_2}V_{y_3,x_1^2}V_{y_3,x_1x_2}V_{y_1,x_2}-30V_{w_2}V_{y_3,x_1^2}V_{y_1,x_2}V_{y_2,x_1^2}+54V_{y_2,x_2}V_{y_3,x_1^2}V_{y_1,x_2}V_{y_2,x_1^2} \\
&+5V_{w_2}V_{y_2,x_1^2}^3V_{y_1,x_2}+8V_{w_2}V_{y_3,x_1x_2}^2V_{y_1,x_2}V_{y_2,x_1^2}-22V_{w_2}V_{y_3,x_1x_2}V_{y_1,x_2}V_{y_2,x_1^2}^2 \\
&+36V_{y_2,x_2}V_{y_3,x_1x_2}V_{y_1,x_2}V_{y_2,x_1^2}^2-6V_{y_3,x_1^2}V_{y_1,x_2}^2V_{y_2,x_1^2}^2+24V_{y_3,x_1^2}V_{y_3,x_1x_2}V_{y_1,x_2}^2V_{y_2,x_1^2} \\
&+\frac{1}{2}V_{y_1,x_2}^2V_{y_2,x_1^2}^4-4V_{y_3,x_1x_2}V_{y_1,x_2}^2V_{y_2,x_1^2}^3+8V_{y_3,x_1x_2}^2V_{y_1,x_2}^2V_{y_2,x_1^2}^2
\end{align*}
}
Here we can directly check $D\circ \varphi=\varphi\circ d$.
Thus we have done from Lemma \ref{dga}.
\end{proof}

\begin{ques}
When $X$ is a homogeneous space, is $Baut_1X$  coformal ?
\end{ques}

\end{document}